\documentclass[a4paper,11pt,reqno]{amsart}
\usepackage{graphicx}
\usepackage{mathtools}
\usepackage{enumerate}
\usepackage{amsmath}
\usepackage{amssymb}
\usepackage{enumitem}
\usepackage{bbm}
\usepackage{tikz}
\usetikzlibrary{positioning}
\usepackage{amsthm}
\usepackage[english]{babel}
\usepackage{booktabs}
\usetikzlibrary{patterns}
\usepackage{cleveref}

\theoremstyle{plain}
\newtheorem{theorem}{Theorem}[section]
\newtheorem{lemma}[theorem]{Lemma}
\newtheorem{proposition}[theorem]{Proposition}
\newtheorem{corollary}[theorem]{Corollary}

\theoremstyle{definition}

\newtheorem{example}[theorem]{Example}

\theoremstyle{remark}
\newtheorem{remark}[theorem]{Remark}

\makeatletter
\@namedef{subjclassname@2020}{%
	\textup{2020} Mathematics Subject Classification}
\makeatother

\newcommand{\Ind}{\big\uparrow}

\newcommand{\F}[1]{\mathbb{F}_{#1}}
\newcommand{\sym}[1]{\mathfrak{S}_{#1}}

\DeclareMathOperator{\Sym}{Sym}
\DeclareMathOperator{\id}{id}
\DeclareMathOperator{\Com}{Com}
\DeclareMathOperator{\m}{\nabla}
\DeclareMathOperator{\mul}{\eta}

\newcommand{\List}[2]{#1_1,#1_2,\dots ,#1_{#2}}

\title{Symmetric powers of $S^{(n-1,1)}$ and $D^{(n-1,1)}$}
\author{Pavel Turek \and Jialin Wang}
\thanks{p.turek@bham.ac.uk, School of Mathematics, Watson Building, University of Birmingham, Edgbaston, Birmingham B15 2TT, United Kingdom}
\thanks{jialin.wang@city.ac.uk, School of Science \& Technology, Department of Mathematics, City St George’s,
University of London, Northampton Square, London EC1V 0HB, United Kingdom}
\date{\today}

\begin{document}	 
\begin{abstract}
Let $p$ be a prime and $n\geq 2$ be a positive integer. We establish new formulae for the decompositions of the first $p-1$ symmetric powers of the Specht module $S^{(n-1,1)}$ and the irreducible module $D^{(n-1,1)}$ in characteristic $p$ as direct sums of Young permutation modules. As an application of the formulae, we show that these symmetric powers have Specht filtration and find the vertices of their indecomposable summands. Our main tool, constructed in this paper, is a lift of a splitting map of a short exact sequence of certain symmetric powers to a splitting map of a short exact sequence of higher symmetric powers. This is a general construction, which can be applied to a broader family of modules. 
\end{abstract}
\maketitle 
	
\thispagestyle{empty}	

\section{Introduction}\label{se:intro}
Let $K$ be a field of prime characteristic $p$, $G$ be a finite group, and $V$ be a $KG$-module. The symmetric powers of $V$ are again $KG$-modules, which can be viewed as quotients of the tensor powers of $V$. Moreover, for $1\leq r<p$, $\Sym^r V$ is a direct summand of $V^{\otimes r}$. In this paper, we focus on symmetric powers of those $KG$-modules $V$ which contain the trivial module, denoted by $K$, as a quotient or a submodule. For such $V$ and any integer $r\geq 1$, there is a surjective map $\partial_r: \Sym^r V \to \Sym^{r-1} V$ and an injective map $\mul_{r-1}: \Sym^{r-1} V \to \Sym^r V$, respectively (see \Cref{ex:trivial} and equations \eqref{eq:Sursequence} and \eqref{eq:Injsequence}, respectively). Our first main result is as follows.

\begin{theorem}\label{th:split}
Let $K$ be a field of prime characteristic $p$, $V$ be a $KG$-module and $r$ be a positive integer which is not congruent to $0$ or $-1$ modulo $p$.
\begin{enumerate}[label=\textnormal{(\roman*)}]
    \item Suppose that there is a surjective map $\varepsilon: V\to K$ of $KG$-modules. If the map $\partial_r: \Sym^r V \to \Sym^{r-1} V$ splits, then so does the map $\partial_{r+1}: \Sym^{r+1} V \to \Sym^r V$.
    
    \item Suppose that there is an injective map $\iota: K\to V$ of $KG$-modules. If the map $\mul_{r-1}: \Sym^{r-1} V \to \Sym^r V$ splits, then so does the map $\mul_r: \Sym^r V \to \Sym^{r+1} V$.
\end{enumerate}
\end{theorem}

\Cref{th:split} is easily applicable in inductive arguments. Indeed, if we show that $\partial_r$ splits for $r=r_0$ not divisible by $p$, then a simple induction on $r$ with conjunction with \Cref{th:split}(i) shows that $\partial_r$ splits for \emph{all} $r_0\leq r\leq (k+1)p-1$ where $kp<r_0\leq (k+1)p-1$. \Cref{th:split}(ii) provides an analogous result for maps $\mul_r$.

It is this consequence of \Cref{th:split} that we apply to permutation modules and to the Specht module $S^{(n-1,1)}$ of the symmetric group $\sym{n}$ for $p\mid n$. For permutation modules, as shown in \Cref{le:permutation}, one can take $r_0=2$ to obtain the following striking result.

\begin{corollary}\label{co:permutation}
    Let $K$ be a field of prime characteristic $p$ and $V$ be a permutation $KG$-module. Then for any $2\leq r\leq p-1$ the surjective map $\partial_r: \Sym^r V \to \Sym^{r-1} V$ splits. In particular, $\Sym^{r-1} V$ is a summand of $\Sym^r V$.
\end{corollary}

We use \Cref{th:split} applied with $V=S^{(n-1,1)}$ (when $p\mid n$) and $V=M^{(n-1,1)}$, the natural permutation module of $\sym{n}$, to prove formulae for symmetric powers $\Sym^r$ of $S^{(n-1,1)}$ (for $2\leq r\leq p-1$) and its simple head $D^{(n-1,1)}$ (for $3\leq r\leq p-1$) in the representation ring of $K\sym{n}$ (see \Cref{pr:ring}). Consequently, we deduce that these symmetric powers decompose as direct sums of indecomposable Young modules; in particular, they have Specht filtration as stated in \Cref{th:SDSpecht}. As further applications of the formulae, we show positivity of several $p$-Kostka numbers in \Cref{co:Kostka} and find vertices of the indecomposable summands of the symmetric powers in question as stated below.

\begin{theorem}\label{th:SDvertex}
    Suppose that $K$ is a field of characteristic $p$ and $n\geq 3$ is an integer divisible by $p$. The indecomposable summands of the symmetric powers $\Sym^r S^{(n-1,1)}$ with $2\leq r\leq p-1$ and of the symmetric powers $\Sym^r D^{(n-1,1)}$ with $3\leq r\leq p-1$ have a vertex given by a Sylow $p$-subgroup of $\sym{n-p}$ or $\sym{n-2p}$. 
\end{theorem}

Observe that the Sylow $p$-subgroups of $\sym{n-p}$ and $\sym{n-2p}$ are \emph{properly} contained in the Sylow $p$-subgroups of $\sym{n}$. Thus, for $r$ as in \Cref{th:SDvertex}, the $r$th symmetric powers of $S^{(n-1,1)}$ and $D^{(n-1,1)}$ are summands of modules induced from proper subgroups of a Sylow $p$-subgroup of $\sym{n}$. This is a surprising result since the same does not hold for the modules $S^{(n-1,1)}$ and $D^{(n-1,1)}$ themselves, as their vertices are the Sylow $p$-subgroups of $\sym{n}$.

\subsection*{Background}

The symmetric powers of $KG$-modules are popular objects of study in invariant theory, due to their connections to polynomial invariants, which, in the language of representation theory, are the elements of the largest subspace of symmetric powers isomorphic to direct sums of copies of the trivial $KG$-module. We direct the reader to \cite{BensonInvariants93} for a detailed exposition of the many known results regarding the polynomial invariants in characteristic $0$.

For a positive characteristic, results about symmetric powers and polynomial invariants applicable to all finite groups are rare. We mention the finiteness result of Karagueuzian and Symonds \cite{DP}, which asserts that for any finite group $G$ and finite field $K$, there are only finitely many non-isomorphic indecomposable summands of symmetric powers of a fixed $KG$-module $V$. We encourage the reader to see the background section of \cite{DP} for detailed references to results about symmetric powers of particular groups in positive characteristic; here we provide only a summary.

Almkvist and Fossum \cite{AlmkvistFossumCyclic78} studied symmetric and exterior powers of modules of a cyclic group of order $p$ in characteristic $p$. Almkvist later applied the results of \cite{AlmkvistFossumCyclic78} to study the number of indecomposable summands of these symmetric powers in \cite{AlmkvistComponents78, AlmkvistReciprocity81}. The work of Almkvist and Fossum was revisited and generalised to groups with Sylow $p$-subgroups of order $p$ by Hughes and Kemper \cite{HughesKemperCyclic00, HughesKemperSylow01}. A careful study of symmetric and exterior powers of the indecomposable modules of SL$_2(\F{p})$ in defining characteristic appears in Kouwenhoven \cite{KouwenhovenLambda90}. An earlier work in a similar direction by Glover \cite{GloverRepresentations78} concerns module structures of symmetric powers of the natural modules of groups SL$_2(\F{p})$, GL$_2(\F{p})$ and the semigroup Mat$_2(\F{p})$. 

The symmetric powers are an example of (polynomial) Schur functors, certain endofunctors of the module categories of groups, which can be used to construct all polynomial modules of general linear groups. A detailed study of the decompositions into indecomposable summands of modules obtained by applying Schur functors to the modules of SL$_2(K)$ in characteristic $0$ has been done by Paget and Wildon \cite{PagetWildonPlethysms21}. For positive characteristic $p$, formulae for computing such decompositions for `small' Schur functors were found by the first author \cite[Theorem~5.3]{TurekStablePlethysms24} following the strategy used in \cite[Section~2.9]{BensonBundles17} to find analogous results for the cyclic group of order $p$.

Our new results about the vertices of the indecomposable summands of the symmetric powers of $S^{(n-1,1)}$ and $D^{(n-1,1)}$ accompany the existing results about the vertices of the exterior powers of the same modules. These exterior powers coincide with modules of the form $S^{(n-k,1^k)}$ and $D^{(n-k, 1^k)}$, vertices of which were initially studied by Murphy and Peel in 1984 \cite{MurphyPeelVertices84}. Since then, the vertices of all the modules $D^{(n-k, 1^k)}$ and most of the modules $S^{(n-k,1^k)}$ were found (see \cite{WildonVertices03, DanzVertices07, MullerZimmermannVertices07, WildonVerticesBlocks10,DanzGiannelliVertices15,GiannelliKayJinWildonVertex16}). Vertices of different modules of symmetric groups, the indecomposable summands of the Foulkes modules, are found in \cite{GiannelliWildonFoulkesandDecomposition15, GiannelliDecomposition15}, and are used to prove new results about the decomposition numbers of symmetric groups.

\subsection*{Outline}

In \Cref{S:sympower} we recall the definition and the structure of symmetric algebras, prove \Cref{th:split} and deduce \Cref{co:permutation}. In Sections~\ref{S:application} and \ref{S:kostka}, we then use \Cref{th:split} to study short exact sequences of modules of the symmetric group $\sym{n}$. In particular, we find formulae for computing symmetric powers of $S^{(n-1,1)}$ and $D^{(n-1,1)}$ in the representation ring of $K\sym{n}$ and deduce results about Specht filtration, vertices (in \Cref{S:application}), and $p$-Kostka numbers (in \Cref{S:kostka}).

\section{Endomorphisms of symmetric powers}\label{S:sympower}
Let $K$ be a field of characteristic $p$, $G$ be a group, and $V$ be a finite-dimensional $KG$-module. We write $\Sym V$ for the graded symmetric $K$-algebra $\bigoplus_{i\geq 0} \Sym^i V$. One can think of $\Sym V$ as a polynomial $K$-algebra: if $\List{x}{t}$ is a basis of $V$, then $\Sym V\cong K[\List{x}{t}]$. 

The symmetric algebra $\Sym V$ has a natural structure of a $KG$-module and a structure of a Hopf algebra. Moreover, the multiplication and comultiplication of this Hopf algebra, which we denote by $\m$ and $\Delta$, respectively, are maps of $KG$-modules. For a non-negative integer $a$, we further write $\Delta_a$ for the composition of the comultiplication followed by the projection to $\Sym^a V \otimes \Sym V$. Thus, $\Delta_a$ restricted to $\Sym^d V$ is the zero map if $d<a$, and its image lies in $\Sym^a V \otimes \Sym^{d-a} V$ if $d\geq a$.

We describe these maps using the identification $\Sym V\cong K[\List{x}{t}]$. To state this description, we need further notation. A \textit{composition of $d$ of length $t$} is a sequence $\alpha = (\List{\alpha}{t})$ of $t$ non-negative integers which add up to $d$. We write $\alpha!$ for $\alpha_1!\alpha_2!\cdots\alpha_t!$, $x^{\alpha}$ for the monomial $x_1^{\alpha_1}x_2^{\alpha_2}\cdots x_t^{\alpha_t}$ in $\Sym V$ and $\partial^{\alpha}$ for the differential operator $\partial^d/(\partial x_1^{\alpha_1}\partial x_2^{\alpha_2}\cdots \partial x_t^{\alpha_t})$ of $\Sym V$. Finally, write $\Com_t(d)$ for the set of compositions of $d$ of length $t$. We use the term \textit{composition of length $t$} to refer to compositions of some non-negative integer $d$ of length $t$.

For compositions $\alpha$ and $\beta$ of length $t$ we denote by $\alpha\pm\beta$ the component-wise sum (or difference) of $\alpha$ and $\beta$. Note that $\alpha+\beta$ is always a composition of length $t$, however, $\alpha-\beta$ may not be a composition of length $t$. The multiplication $\m$ then sends $x^{\alpha}\otimes x^{\beta}$ to $x^{\alpha + \beta}$. Writing $\alpha \leq \beta$ when $\alpha_i\leq \beta_i$ for all $i\leq t$, the map $\Delta_a$ sends $x^{\beta}$ to $\sum_{\substack{\alpha\in \Com_t(a)\\ \alpha\leq \beta}} \binom{\beta_1}{\alpha_1}\binom{\beta_2}{\alpha_2}\cdots \binom{\beta_t}{\alpha_t}x^{\alpha}\otimes x^{\beta-\alpha}$.

The following construction is the main tool used in this paper. Given a homomorphism $\phi:\Sym^a V\to \Sym^b V$ we define its \textit{degree-less lift} $\Psi(\phi)$ as the endomorphism of the $KG$-module $\Sym V$ which acts on the submodule $\Sym^d V$ as zero if $d<a$, and as the following composition
\[
\Sym^d V \xrightarrow{\Delta_a} \Sym^a V\otimes \Sym^{d-a} V \xrightarrow{\phi\otimes \id} \Sym^b V\otimes \Sym^{d-a} V \xrightarrow{\m} \Sym^{d+b-a} V
\]
if $d\geq a$. From the definition, it is clear that $\Psi(\phi)\vert _{\Sym^a V} = \phi$.

\begin{example}\label{ex:identity}
For $\phi=\id : \Sym^a V \to \Sym^a V$, we get that $\Psi(\phi) = \m\circ \Delta_a$. Thus, $\Psi(\phi)$ acts on $\Sym^d A$ as a scalar multiplication by $\binom{d}{a}$: indeed $\Psi(\phi)$ sends $x^{\beta}\in \Sym^d V$ to $\sum_{\substack{\alpha\in \Com_t(a)\\ \alpha\leq \beta}} \binom{\beta_1}{\alpha_1}\binom{\beta_2}{\alpha_2}\cdots \binom{\beta_t}{\alpha_t}x^{\beta} = \binom{d}{a}x^{\beta}$, where the equality follows from the argument that choosing $a$ objects out of $d$ objects is the same as choosing $\alpha_1$ objects from the first $\beta_1$ objects, $\alpha_2$ objects from the next $\beta_2$ objects and so on, for some $\alpha\in \Com_t(a)$ such that $\alpha\leq \beta$.
\end{example}

The isomorphism $\Sym V\cong K[\List{x}{t}]$ provides us with a practical description of the degree-less lift.

\begin{lemma}\label{le:differential form}
Let $\phi: \Sym^a V\to \Sym^b V$ be a homomorphism of $KG$-modules. After the identification $\Sym V\cong K[\List{x}{t}]$, we obtain that $\Psi(\phi) = \sum_{\alpha\in \Com_t(a)} \phi(x^{\alpha}) \partial^{\alpha}/\alpha!$.
\end{lemma}

\begin{remark}\label{re:welldefined}
If $\alpha!=0$ in $K$, we interpret the map $\partial^{\alpha}/\alpha!$ as the map which sends monomial $x^{\beta}$ to $\binom{\beta_1}{\alpha_1}\binom{\beta_2}{\alpha_2}\cdots \binom{\beta_t}{\alpha_t} x^{\beta-\alpha}$ if $\beta\geq\alpha$ and to zero otherwise. Note that $\partial^{\alpha}/\alpha!$ has the same effect even if $\alpha! \neq 0$. 
\end{remark}

\begin{proof}
Both maps, $\Psi(\phi)$ and $\sum_{\alpha\in \Com_t(a)} \phi(x^{\alpha}) \partial^{\alpha}/\alpha!$, vanish on polynomials of degree less than $a$. From the definition of $\Psi(\phi)$ and the above discussion, for any $\beta\in\Com_t(d)$ with $d\geq a$, the monomial $x^{\beta}$ is sent by $\Psi(\phi)$ to
\[\sum_{\substack{\alpha\in \Com_t(a)\\ \alpha\leq \beta}} \binom{\beta_1}{\alpha_1}\binom{\beta_2}{\alpha_2}\cdots \binom{\beta_t}{\alpha_t} x^{\beta-\alpha} \phi(x^{\alpha}).\]
This coincides with $\sum_{\alpha\in \Com_t(a)} \phi(x^{\alpha}) \partial^{\alpha} x^{\beta}/\alpha!$ (see \Cref{re:welldefined}), finishing the proof.
\end{proof}

We apply \Cref{le:differential form} to two families of modules which form the main focus of the remainder of the paper. In the examples and the rest of the paper, we write $K$ for the trivial $KG$-modules.

\begin{example}\label{ex:trivial}
Let $V$ be a $KG$-module of dimension $t$.
\begin{enumerate}[label=\textnormal{(\roman*)}]
    \item Suppose that there is a surjective map $\varepsilon: V \to K$ of $KG$-modules. Let $\List{x}{t}$ be a basis of $V$ such that $\varepsilon(x_i)=1$ for all $1\leq i\leq t$. By \Cref{le:differential form}, the endomorphism $\partial := \Psi(\varepsilon)$ of the $KG$-module $\Sym V \cong K[\List{x}{t}]$ equals the sum of differential operators $\sum_{i\leq t} \partial/\partial x_i$.

    \item Suppose that there is an injective map $\iota: K \to V$ of $KG$-modules. Let $\List{e}{t}$ be a basis of $V$ such that $\iota(1)=\sum_{i\leq t} e_i$. By \Cref{le:differential form}, the endomorphism $\mul := \Psi(\iota)$ of the $KG$-module $\Sym V \cong K[\List{e}{t}]$ equals the multiplication by $\sum_{i\leq t} e_i$.
\end{enumerate}
\end{example}

Continuing with \Cref{ex:trivial}(i), let $W$ be the kernel of $\varepsilon$. The inclusion of $W$ in $V$ gives rise to a degree-preserving injective map $\Sym W\to \Sym V$ of $KG$-modules. Using \cite[Lemma~3]{BensonKayJin14SymExterior} or direct calculations, one obtains a short exact sequence

\begin{equation}\label{eq:Sursequence}
0\to \Sym^r W\to \Sym^r V \xrightarrow{\partial_r} \Sym^{r-1} V\to 0
\end{equation}
for any $r\geq 1$, where $\partial_r$ is the map $\partial$ from \Cref{ex:trivial}(i) restricted to $\Sym^r V$.

Similarly, if $W$ is the cokernel of $\iota$ in \Cref{ex:trivial}(ii), then there is a degree-preserving surjective map $\Sym V \to \Sym W$ and by \cite[Lemma~3]{BensonKayJin14SymExterior} one obtains a short exact sequence

\begin{equation}\label{eq:Injsequence}
0\to \Sym^{r-1} V \xrightarrow{\mul_{r-1}} \Sym^r V \to \Sym^r W \to 0 
\end{equation}
for any $r\geq 1$, where $\mul_{r-1}$ is the map $\mul$ from \Cref{ex:trivial}(ii) restricted to $\Sym^{r-1} V$. While the short exact sequences \eqref{eq:Sursequence} and \eqref{eq:Injsequence} seem dual to each other, this is only the case for $r<p$. Thus, it is necessary to treat both sequences separately.

To prove our first main result, we use the following computational lemma.

\begin{lemma}\label{le:multiplication}
Let $r$ be a positive integer and $V$ a finite-dimensional $KG$-module.
\begin{enumerate}[label=\textnormal{(\roman*)}]
    \item Suppose that there exists a surjection $\varepsilon: V \to K$ and let $\partial_r$ be the restriction of $\partial=\Psi(\varepsilon)$ to $\Sym^r V$. If $\phi:\Sym^{r-1} V\to \Sym^r V$ is a map of $KG$-modules such that $\partial_r\circ \phi = \id$, then the commutator $[\partial, \Psi(\phi)] : \Sym V \to \Sym V$ acts on $\Sym^d V$ by multiplication by $\binom{d}{r-1}$.
    \item Suppose that there exists an injection $\iota: K \to V$ and let $\mul_{r-1}$ be the restriction of $\mul = \Psi(\iota)$ to $\Sym^{r-1} V$. If $\phi:\Sym^r V\to \Sym^{r-1} V$ is a map of $KG$-modules such that $\phi\circ \mul_{r-1} = \id$, then the commutator $[\Psi(\phi), \mul] : \Sym V \to \Sym V$ acts on $\Sym^d V$ by multiplication by $\binom{d}{r-1}$.
\end{enumerate}
\end{lemma}

\begin{proof}
Recall from \Cref{ex:trivial}(i) that $\partial = \sum_{i\leq t} \partial/\partial x_i$, where $\List{x}{t}$ is a basis of $V$ as in \Cref{ex:trivial}(i). By \Cref{le:differential form} and the product rule, we see that in (i)
\[ [\partial, \Psi(\phi)] = \sum_{\alpha\in\Com_t(r-1)}\partial(\phi(x^{\alpha})) \frac{\partial^{\alpha}}{\alpha!}. \]
For $\alpha\in\Com_t(r-1)$, let $f_{\alpha} = \partial(\phi(x^{\alpha})) = [\partial, \Psi(\phi)] (x^{\alpha})$, which is a polynomial in $K[\List{x}{t}]$ of degree $r-1$. Since $\Psi(\phi)$ vanishes on $\Sym^{r-2} V$ and acts as $\phi$ on $\Sym^{r-1} V$, from $\partial_r\circ \phi = \id$, we conclude that $[\partial, \Psi(\phi)]$ restricted to $\Sym^{r-1} V$ becomes the identity map. Thus $f_{\alpha} =  x^{\alpha}$ for any composition $\alpha\in\Com_t(r-1)$. In turn, by \Cref{le:differential form}, $[\partial, \Psi(\phi)] = \sum_{\alpha\in\Com_t(r-1)}x^{\alpha} \frac{\partial^{\alpha}}{\alpha!}$ is the degree-less lift of the identity map on $\Sym^{r-1} V$. The result now follows from \Cref{ex:identity}.

The situation in (ii) is analogous. Recall from \Cref{ex:trivial}(ii) that $\mul$ is the multiplication by $\sum_{i\leq t} e_i$, where $\List{e}{t}$ is a basis of $V$ as in \Cref{ex:trivial}(ii). One then computes
\[ [\Psi(\phi), \mul] = \sum_{\beta\in\Com_t(r)} \phi(e^{\beta}) \sum_{\substack{i\leq t\\ \beta_i\geq 1}} \beta_i \frac{\partial^{\beta - \delta^i}}{\beta!} = \sum_{\alpha\in\Com_t(r-1)} \left(\sum_{i\leq t}\phi(e_ie^{\alpha})\right) \frac{\partial^{\alpha}}{\alpha!}, \]
where $\delta^i$ is the composition of $1$ of length $t$ with $\delta_i^i = 1$.

Similarly as above, $\Psi(\phi)$ vanishes on $\Sym^{r-1} S$ and acts as $\phi$ on $\Sym^r S$, and $\phi\circ \mul_{r-1} = \id$; thus $[\Psi(\phi), \mul]$ restricted to $\Sym^{r-1} S$ is the identity. The result follows as in (i). 
\end{proof}

We are now ready to prove the main result.

\begin{proof}[Proof of \Cref{th:split}]
Let $\phi:\Sym^{r-1} V \to \Sym^r V$ be the splitting of $\partial_r$ in (i). We will show that
\[\theta = \frac{1}{r}\left( \Psi(\phi) - \frac{\Psi(\phi)^2\circ\partial}{r+1}\right),\]
restricted to $\Sym^r V$ is the desired splitting. Using \Cref{le:multiplication}(i), for $v\in \Sym^d V$ we compute

\begin{align*}
&r \partial\circ\theta(v) \\
=\;& [\partial, \Psi(\phi)](v) + \Psi(\phi)\circ \partial(v)\\
&- \frac{[\partial, \Psi(\phi)]\circ\Psi(\phi)\circ \partial(v) + \Psi(\phi)\circ[\partial, \Psi(\phi)]\circ \partial(v) + \Psi(\phi)^2\circ\partial^2 (v)}{r+1}\\
=\;& \binom{d}{r-1} v + \left( 1- \frac{\binom{d}{r-1} + \binom{d-1}{r-1}}{r+1} \right)\Psi(\phi)\circ\partial (v) - \frac{\Psi(\phi)^2\circ\partial^2 (v)}{r+1}.   
\end{align*}
In particular, if $d=r$, then the coefficient of $v$ becomes $r$, the coefficient of $\Psi(\phi)\circ\partial (v)$ becomes $0$ and $\Psi(\phi)^2\circ\partial^2 (v)=0$ as $\Psi(\phi)$ vanishes on polynomials of degree $r-2$. Thus $\partial\circ \theta(v) = v$ for any $v\in \Sym^r V$, as required.

Now suppose that $\phi:\Sym^r V \to \Sym^{r-1} V$ is a splitting of $\mul_{r-1}$ in (ii). This time, we show that
\[\theta = \frac{1}{r}\left(\Psi(\phi) - \frac{\mul\circ\Psi(\phi)^2}{r+1}\right),\]
restricted to $\Sym^{r+1} V$ is the desired splitting. Using \Cref{le:multiplication}(ii), for $v\in \Sym^d V$ we compute

\begin{align*}
&r \theta\circ\mul(v)\\ =\;& [\Psi(\phi),\mul](v) + \mul\circ\Psi(\phi)(v)\\
&- \frac{\mul\circ\Psi(\phi)\circ[\Psi(\phi),\mul](v) + \mul\circ[\Psi(\phi),\mul]\circ \Psi(\phi)(v) + \mul^2\circ\Psi(\phi)^2(v)}{r+1}\\
=\;& \binom{d}{r-1} v + \left( 1- \frac{\binom{d}{r-1} + \binom{d-1}{r-1}}{r+1} \right)\mul\circ\Psi(\phi) (v) - \frac{\mul^2\circ\Psi(\phi)^2 (v)}{r+1}.   
\end{align*}
As $\Psi(\phi)^2$ vanishes on polynomials of degree $r$, similarly to (i), we conclude that $\theta\circ\mul(v)=v$ for any $v\in \Sym^r V$, as required.
\end{proof}

We end this section by applying \Cref{th:split} to permutation modules. Recall that a $KG$-module $V$ is a \textit{permutation module} if there is a basis $\List{x}{t}$ of $V$ which is invariant under the action of $G$ on $V$. The basis $\List{x}{t}$ is referred to as a \textit{permutation basis} of $V$. Up to scalar multiplication, there is a unique surjective map $\varepsilon: V \to K$, which is given by $\varepsilon(x_i)=1$ for all $i$ and a unique injective map $\iota: K \to V$, which is given by $\iota(1) = \sum_{i\leq t} x_i$. In turn, $V=\Sym^1 V$ has $K=\Sym^0 V$ as a summand if and only if $\varepsilon\circ\iota \neq 0$, which happens precisely when $p\nmid \dim V$. 

However, for any permutation module $V$, $\Sym^r V$ has $\Sym^{r-1} V$ as a summand if $2\leq r\leq p-1$ as stated in \Cref{co:permutation}. Note that for $V\not\cong K$ the composition $\partial_r\circ \mul_{r-1}$ is not a scalar multiplication of the identity map, and thus to show this fact, a new map is needed.

\begin{lemma}\label{le:permutation}
    Let $V$ be a permutation $KG$-module with a permutation basis $\List{x}{t}$. The map $\zeta: V\to \Sym^2 V$ given by $\zeta(x_i) =x_i^2/2$ for any $i\leq t$ is a splitting of the map $\partial_2: \Sym^2 V \to V$. 
\end{lemma}

\begin{proof}
    Since $G$ permutes $\List{x}{t}$, the map $\zeta$ is a homomorphism of $KG$-modules. One computes that $\partial_2\circ \zeta (x_i) = x_i$ for all $i\leq t$, which finishes the proof.
\end{proof}

We immediately conclude the promised result about permutation modules.

\begin{proof}[Proof of \Cref{co:permutation}]
    We prove the statement by induction on $r$. If $r=2$, then we can use \Cref{le:permutation}. If $3\leq r\leq p-1$ and the statement holds for $r-1$, then we obtain the desired splitting from \Cref{th:split}(i) since $r-1$ is not congruent to $0$ or $-1$ modulo $p$.
\end{proof}

\section{Applications to $S^{(n-1,1)}$ and $D^{(n-1,1)}$}\label{S:application}

\subsection{Background on the symmetric group}

We recall the necessary background from the representation theory of symmetric groups for this paper. For a more detailed account of this area, see, for instance, \cite{JamesSymmetric78}. 

Let $n$ be a non-negative integer. A \textit{partition} $\lambda = (\List{\lambda}{l})$ of $n$ is a non-increasing sequence of positive integers, called \textit{parts}, which add up to $n$. The \textit{length} of $\lambda$ denoted by $\ell(\lambda)$ is the number of parts of $\lambda$, and the \textit{size} of $\lambda$ denoted by $|\lambda|$ is $n$. We use the convention that $\lambda_i = 0$ for any $i>\ell(\lambda)$. A partition is \textit{$p$-regular} if each of its entries repeats at most $(p-1)$-times. And it is \textit{$p$-restricted} if for any positive integer $i$, we have $\lambda_i-\lambda_{i+1} \leq p-1$. For partitions $\lambda$ and $\mu$ of $n$, we say that $\lambda$ \textit{dominates} $\mu$ if for all positive integers $j$ we have $\sum_{i\leq j} \lambda_i \geq \sum_{i\leq j} \mu_i$. If this is the case, we write $\lambda \unrhd \mu$.

Let $\sym{n}$ denote the symmetric group indexed by $n$ and for a partition $\lambda$ of $n$, let $\sym{\lambda}$ be its \textit{Young subgroup} $\sym{\lambda_1}\times \sym{\lambda_2}\times\dots\times \sym{\lambda_{\ell(\lambda)}}$. We fix an underlying field $K$. The \textit{Young permutation module} $M^{\lambda}$ is the permutation $K\sym{n}$-module with a permutation basis given by the left cosets of $\sym{\lambda}$ (permuted naturally by $\sym{n}$). When the underlying field $K$ has characteristic $0$, the irreducible $K\sym{n}$-modules are given by the \textit{Specht modules} $S^{\lambda}$ indexed by partitions $\lambda$ of $n$. The Specht module $S^{\lambda}$ is defined as a submodule of $M^{\lambda}$ with a characteristic free basis given by the set of $\lambda$-polytabloids. Moreover, the Specht modules are indecomposable if $K$ has positive characteristic $p>2$. Suppose now $K$ is a field of positive characteristic $p$. The irreducible $K\sym{n}$-modules are the (irreducible) heads of Specht modules labelled by $p$-regular partitions of $n$. Given a $p$-regular partition $\lambda$, we denote the head of $S^\lambda$ by $D^{\lambda}$.

The final family of $K\sym{n}$-modules we need to introduce are the indecomposable \textit{Young modules} $Y^{\lambda}$ indexed by partitions of $n$. James \cite[Theorem~3.1]{JamesTrivial83} showed that every Young permutation module $M^{\lambda}$ decomposes as a sum of Young modules $Y^{\mu}$ with $\mu\unrhd \lambda$ and that the Young module $Y^{\lambda}$ appears with multiplicity one.

Throughout we assume $n\geq 2$. We mainly consider modules labelled by the partition $(n-1,1)$. The Young permutation module $M^{(n-1,1)}$ is the natural $n$-dimensional permutation $K\sym{n}$-module. We denote its permutation basis by $\List{x}{n}$. The surjection $M^{(n-1,1)} \to K$ gives rise to a (characteristic-independent) short exact sequence
\[
0 \to S^{(n-1,1)} \to M^{(n-1,1)} \to K \to 0.
\]
In the spirit of \eqref{eq:Sursequence}, for any $r\geq 1$ we obtain a short exact sequence
\begin{equation}\label{eq:Msequence}
0\to \Sym^r S^{(n-1,1)}\to \Sym^r M^{(n-1,1)} \xrightarrow{\partial_r} \Sym^{r-1} M^{(n-1,1)}\to 0.
\end{equation}
Suppose that $n\geq 3$. The elements $e_i=x_i-x_n$ with $i\leq n-1$ of the Specht module $S^{(n-1,1)}\leq M^{(n-1,1)}$ form its basis. This Specht module is irreducible if and only if the characteristic $p$ of the underlying field $K$ does not divide $n$ (in which case $D^{(n-1,1)} = S^{(n-1,1)}$). If $p\mid n$, then there is a short exact sequence
\[
0 \to K \to S^{(n-1,1)} \to D^{(n-1,1)} \to 0,
\]
where the injection $K\to S^{(n-1,1)}$ sends $1$ to $\sum_{i\leq n-1} e_i$. As in \eqref{eq:Injsequence}, for any $r\geq 1$ we obtain a short exact sequence
\begin{equation}\label{eq:Ssequence}
0\to \Sym^{r-1} S^{(n-1,1)} \xrightarrow{\mul_{r-1}} \Sym^r S^{(n-1,1)} \to \Sym^r D^{(n-1,1)} \to 0.
\end{equation}

\subsection{Symmetric powers of $S^{(n-1,1)}$ and $D^{(n-1,1)}$}

We have already seen that \eqref{eq:Msequence} splits for suitable $r$.

\begin{lemma}\label{le:Msplits}
    Let $K$ be a field of positive characteristic $p$ and $n\geq 2$ be an integer. The short exact sequence \eqref{eq:Msequence} of $K\sym{n}$-modules splits for $2\leq r\leq p-1$.
\end{lemma}

\begin{proof}
    This is \Cref{co:permutation} for $V=M^{(n-1,1)}$.
\end{proof}

\begin{remark}\label{re:Msplit}
    As mentioned before \Cref{le:permutation}, the short exact sequence \eqref{eq:Msequence} with $r=1$ splits if and only if $p$ does not divide the dimension of $M^{(n-1,1)}$, that is, if $p\nmid n$. In that case $M^{(n-1,1)}\cong S^{(n-1,1)}\oplus K$. It is well-known that if $M\cong N\oplus L$, then 
    $$\Sym^rM\cong \bigoplus_{0\leq q\leq r}\Sym^qN\otimes \Sym^{r-q}L.$$
    Thus, we have 
    $$\Sym^rM^{(n-1,1)}\cong \bigoplus_{0\leq q\leq r}\Sym^qS^{(n-1,1)}\otimes \Sym^{r-q}K\cong \bigoplus_{0\leq q\leq r}\Sym^qS^{(n-1,1)}.$$
    In particular, $\Sym^rM^{(n-1,1)}\cong \Sym^rS^{(n-1,1)}\oplus \Sym^{r-1}M^{(n-1,1)}$ and \Cref{le:Msplits} follows immediately in this case.
\end{remark}

To deduce an analogous result for \eqref{eq:Ssequence}, we need an initial splitting to apply \Cref{th:split}(ii) to.

\begin{lemma}\label{le:Smap}
Suppose that $K$ is a field of odd characteristic $p$, and $n$ is a positive integer divisible by $p$. The map $\gamma:\Sym^3 S \to \Sym^2 S$ which is given by 
\begin{align*}
    e_i^3&\mapsto -\frac{1}{2}e_i \sum_{l\leq n-1} e_l, \\
    e_i^2e_j&\mapsto \frac{1}{2}(e_ie_j-e_i^2-e_j^2) -\frac{1}{4} \sum_{l\leq n-1} e_l^2, \\
    e_ie_je_k&\mapsto -\frac{1}{4} (e_i^2 +e_j^2 +e_k^2 +\sum_{l\leq n-1} e_l^2)
\end{align*}
(for pairwise distinct $i,j$ and $k$) is a $K\sym{n}$-homomorphism which satisfies $\gamma\circ \mul_2 = id$.
\end{lemma}

\begin{proof}
This is verified in \Cref{se:proof}.
\end{proof}

We are now ready to apply \Cref{th:split}(ii).

\begin{corollary}\label{cor:Ssplits}
    Suppose that $K$ is a field of odd characteristic $p$ and $n$ is a positive integer divisible by $p$. The short exact sequence \eqref{eq:Ssequence} of $K\sym{n}$-modules splits for $3\leq r\leq p-1$.
\end{corollary}

\begin{proof}
    Note that $p$ is at least $5$ as it is a prime and $3\leq p-1$. The result follows by induction on $r$ using \Cref{le:Smap} and \Cref{th:split}.
\end{proof}

\begin{remark}\label{re:Ssplit}
    In the setting of \Cref{cor:Ssplits}, the short exact sequence \eqref{eq:Ssequence} does not split if $r=1$ or $r=2$. The case $r=1$ is well-known: $S^{(n-1,1)}$ is an indecomposable module if $p\mid n$. For $r=2$ one can use that for $n>3$, there is an isomorphism $\Sym^2 S^{(n-1,1)} \cong M^{(n-2,2)}$ which has two indecomposable summands -- $M^{(n-1,1)}$ of dimension $n$ and the other of dimension $n(n-3)/2$ (see \Cref{ex:SandD} and \Cref{S:kostka}), neither of which is isomorphic to $S^{(n-1,1)}$ which has dimension $n-1$. If $n=3$, then $\Sym^2 S^{(2,1)}\cong M^{(2,1)}$ is indecomposable. Hence \eqref{eq:Ssequence} does not split for $r=2$. 
\end{remark}

We can immediately deduce new formulae for the symmetric powers of $S^{(n-1,1)}$ and $D^{(n-1,1)}$ in the representation ring. Throughout the rest of the paper, $K$ is a field with positive characteristic $p$ and for a $K\sym{n}$-module $V$, we write $[V]$ for the corresponding element in the representation ring of $K\sym{n}$.

\begin{proposition}\label{pr:ring}
    Suppose that $n\geq 2$ is an integer. In the representation ring of $K\sym{n}$, we have 
    \[
    [\Sym^r S^{(n-1,1)}] = [\Sym^r M^{(n-1,1)}] - [\Sym^{r-1} M^{(n-1,1)}]
    \]
    for $2\leq r\leq p-1$. If moreover $p\mid n$, then 
    \[
    [\Sym^r D^{(n-1,1)}] = [\Sym^r M^{(n-1,1)}] + [\Sym^{r-2} M^{(n-1,1)}] - 2[\Sym^{r-1} M^{(n-1,1)}]
    \]
    for $3\leq r\leq p-1$. 
\end{proposition}

\begin{proof}
    The first formula follows directly from \Cref{le:Msplits}. From \Cref{cor:Ssplits}, we have
    \[
    [\Sym^r D^{(n-1,1)}] = [\Sym^r S^{(n-1,1)}] - [\Sym^{r-1} S^{(n-1,1)}]
    \]
    for any $3\leq r\leq p-1$. The second formula now follows from the first one applied with $r$ and $r-1$.
\end{proof}

As the next step, we decompose the symmetric powers of $M^{(n-1,1)}$ as direct sums of Young permutation modules. To do so, for a partition $\lambda$ and a non-negative integer $r$, let $y_r^{\lambda}$ be the number of sequences $d=(d_0,d_1,\dots)$ of non-negative integers such that (a) the non-zero entries of $d$ are (in some order) the parts of $\lambda$, and (b) $\sum_{i\geq 0} id_i = r$.

\begin{lemma}\label{le:Mdecomposition}
    Let $r$ be a non-negative integer. There is an isomorphism $\Sym^r M^{(n-1,1)} \cong \bigoplus_{\lambda} y_r^{\lambda} M^{\lambda}$, where the sum runs over all partitions of $n$.
\end{lemma}

\begin{proof}
    Consider $\Sym^r M^{(n-1,1)}$ as the vector space of homogeneous polynomials in $\List{x}{n}$ of degree $r$. For each sequence $d=(d_0,d_1,\dots)$ of non-negative integers such that $\sum_{i\geq 0}d_i = n$ and $\sum_{i\geq 0} id_i = r$, let $V_d$ be the subspace of $\Sym^r M^{(n-1,1)}$ spanned by monomials $x^{\alpha}$ where the composition $\alpha$ has $d_i$ entries equal to $i$ for any $i\geq 0$. It is easy to see that $\Sym^r M^{(n-1,1)}$ decomposes as the direct sum of $V_d$, where $d$ runs over all sequences of non-negative integers such that $\sum_{i\geq 0}d_i = n$ and $\sum_{i\geq 0} id_i = r$, and that $V_d$ is in fact a $K\sym{n}$-submodule of $\Sym^r M^{(n-1,1)}$ isomorphic to $M^{\lambda}$, where $\lambda$ is obtained from $d$ by ordering its non-zero entries. The result follows.
\end{proof}

As a consequence, in the representation ring of $K\sym{n}$, we have the formulae
\begin{equation}\label{eq:Sformula}
    [\Sym^r S^{(n-1,1)}] = \sum_{\lambda} (y_r^{\lambda} - y_{r-1}^{\lambda}) [M^{\lambda}]
\end{equation}
for $2\leq r\leq p-1$ and if $p\mid n$,
\begin{equation}\label{eq:Dformula}
    [\Sym^r D^{(n-1,1)}] = \sum_{\lambda} (y_r^{\lambda} + y_{r-2}^{\lambda} - 2y_{r-1}^{\lambda}) [M^{\lambda}]
\end{equation}
for $3\leq r\leq p-1$, where both sums run over all partitions of $n$.

\begin{example}\label{ex:SandD}
    Suppose that $p\geq 5$. 
    We have that $y_1^{\lambda}$ is zero unless $\lambda = (n-1,1)$, in which case it is one. Similarly, we compute that $y_r^{\lambda}$ with $r=2,3,4$ (and $n\geq 2r)$ takes only values $0$ and $1$ and the parameters for which it equals $1$ are in \Cref{tab:coefficients}. Using \eqref{eq:Sformula} and \eqref{eq:Dformula}, we thus find
    \begin{align*}
        [\Sym^2 S^{(n-1,1)}] &= [M^{(n-2,2)}] \textnormal{ (valid even for $p=3$)},\\
        [\Sym^3 S^{(n-1,1)}] &= [M^{(n-2,1^2)}] + [M^{(n-3,3)}] - [M^{(n-2,2)}],\\
        [\Sym^4 S^{(n-1,1)}] &= [M^{(n-2,2)}] + [M^{(n-3,2,1)}] + [M^{(n-4,4)}] - [M^{(n-3,3)}],\\
        [\Sym^3 D^{(n-1,1)}] &= [M^{(n-2,1^2)}] + [M^{(n-3,3)}] - 2[M^{(n-2,2)}],\\
        [\Sym^4 D^{(n-1,1)}] &= 2[M^{(n-2,2)}] + [M^{(n-3,2,1)}] + [M^{(n-4,4)}] - [M^{(n-2,1^2)}] - 2[M^{(n-3,3)}].
    \end{align*}
\end{example}

\begin{table}[h]
    \centering
    \begin{tabular}{cc}
    \toprule
         $r$ & $\lambda$ such that $y_r^{\lambda}=1$ \\
         \midrule
         $2$ & $(n-1,1), (n-2,2)$\\
         \midrule
         $3$ & $(n-1,1), (n-2,1^2), (n-3,3)$ \\
         \midrule
         $4$ & $(n-1,1), (n-2,2), (n-2,1^2), (n-3,2,1), (n-4,4)$\\
         \bottomrule
    \end{tabular}
    \vspace{4pt}
    \caption{The values of $r=2,3,4$ and $\lambda$ for which $y_r^{\lambda}=1$. For other parameters in this range $y_r^{\lambda}=0$.}
    \label{tab:coefficients}
\end{table}

Recall that the Young permutation modules decompose as a direct sum of the (indecomposable) Young modules. From \eqref{eq:Sformula}, working over a field of characteristic $p$, the symmetric powers $\Sym^r S^{(n-1,1)}$ are sums of Young modules for $2\leq r\leq p-1$. Moreover if $p\mid n$, then so are the symmetric powers $\Sym^r D^{(n-1,1)}$ for $3\leq r\leq p-1$. The following result about Specht filtration is due to Donkin \cite[(2.6)]{DonkinSchurAlgebra87}.

\begin{theorem}\label{th:YSpecht}
    The Young modules have Specht filtration.
\end{theorem}

We immediately deduce the existence of a Specht filtration for our symmetric powers.

\begin{theorem}\label{th:SDSpecht}
    Suppose that $K$ is a field of characteristic $p$ and $n\geq 2$ is an integer. The symmetric power $\Sym^r S^{(n-1,1)}$ has a Specht filtration for $0\leq r\leq p-1$. Moreover if $p\mid n$, then the symmetric power $\Sym^r D^{(n-1,1)}$ has a Specht filtration for $3\leq r\leq p-1$. 
\end{theorem}

\begin{proof}
    The statement is clear for $\Sym^r S^{(n-1,1)}$ with $r\leq 1$. From \eqref{eq:Sformula} and \eqref{eq:Dformula}, the other symmetric powers in question are sums of Young modules, and thus have Specht filtration by \Cref{th:YSpecht}.
\end{proof}

\begin{remark}\label{re:SmallSpecht}
    If $n\geq 4$, the irreducible module $D^{(n-1,1)}$ does not have a Specht filtration. However, while not included in \Cref{th:SDSpecht}, the symmetric square $\Sym^2 D^{(n-1,1)}$ has a Specht filtration -- as mentioned in \Cref{re:Ssplit}, the modules $\Sym^2 S^{(n-1,1)}$ and $M^{(n-2,2)}$ are isomorphic (see also \Cref{ex:SandD}) and by \Cref{pkostka1}, $M^{(n-1,1)}\cong Y^{(n-1,1)}\oplus Y^{(n-2,2)}$. Furthermore, $Y^{(n-1,1)}\cong M^{(n-1,1)}$ has dimension $n$ and $Y^{(n-2,2)}\cong S^{(n-2,2)}$ (see \cite[Proposition~5.40]{MathasHecke99} and \cite[Proposition~1.1]{HEMMER2006433}) has dimension $\frac{n(n-3)}{2}$.
    Hence $\Sym^2 S^{(n-1,1)} \cong M^{(n-1,1)} \oplus S^{(n-2,2)}$. Using the short exact sequence \eqref{eq:Ssequence} with $r=2$ and the structure of $M^{(n-1,1)}$, we conclude that $\Sym^2 D^{(n-1,1)} \cong K \oplus S^{(n-2,2)}$, which clearly has a Specht filtration.
\end{remark}

\subsection{Vertices}

Let $G$ be a finite group and $V$ an indecomposable $KG$-module. A \textit{vertex} of $V$ is a minimal subgroup $H\leq G$ for which there is an indecomposable $KH$-module $U$, called a \textit{source} of $V$, such that $V$ is an indecomposable summand of the induced module $U\Ind_H^G$. We note that if $K$ has positive characteristic $p$, then all vertices of $V$ are $p$-groups.

Any indecomposable summand of a permutation module has a trivial source. Thus, from \eqref{eq:Sformula} and \eqref{eq:Dformula}, we see that all indecomposable summands of $\Sym^r S^{(n-1,1)}$ with $2\leq r\leq p-1$ and, also of $\Sym^r D^{(n-1,1)}$ with $p\mid n$ and $3\leq r\leq p-1$ have a trivial source. The task of describing the vertices of these indecomposable summands is less straightforward; in particular, we need some simple observations about the coefficients $y_r^{\lambda}$.

\begin{lemma}\label{le:coefficients}
    Let $r<n$ be positive integers. If $\lambda$ is a partition of $n$ such that either
    \begin{enumerate}[label=\textnormal{(\roman*)}]
        \item $y_r^{\lambda}$ is non-zero, or
        \item $r\geq 2$ and $y_r^{\lambda} - y_{r-1}^{\lambda}$ is non-zero, or
        \item $r\geq 3$ and $y_r^{\lambda} + y_{r-2}^{\lambda} - 2y_{r-1}^{\lambda}$ is non-zero,
    \end{enumerate}
    then $n-r\leq \lambda_1 <n$.
\end{lemma}

\begin{proof}
    Suppose firstly that $y_r^{\lambda}>0$. Then there is a sequence $d=(d_0,d_1,\dots )$ of non-negative integers such that (a) the non-zero entries of $d$ are (in some order) the entries of $\lambda$, and (b) $\sum_{i\geq 0} id_i = r$. Thus we have $d_0 = n-\sum_{i\geq 1} d_i \geq n-\sum_{i\geq 0} id_i = n-r$  Hence $\lambda_1\geq n-r$. On the other hand, we cannot have $\lambda_1=n$, as otherwise there would be $i\geq 0$ such that $d_i=n$ and all remaining entries of $d$ would be zero; thus $in=r$, which is impossible as $0<r<n$.
    
    Hence, if (i) holds, then $n-r\leq \lambda_1 <n$. As $n-r<n-(r-1)<n-(r-2)$, we deduce the same if (ii) or (iii) holds.
\end{proof}

We will need the description of the vertices of the Young modules, which can be found in \cite[Theorem~4.6.3]{MartinSchur93} (due to Grabmeier \cite{GrabmeierYoung85} and Klyachko \cite{KlyachkoSummands83}). For any partition $\lambda$, the \textit{$p$-adic expansion} of $\lambda$ is a sequence of $p$-restricted partitions $(\lambda(0), \lambda(1), \dots )$ such that for any positive integer $i$, we have $\lambda_i = \sum_{j\geq 0} \lambda(j)_i p^j$. As with the $p$-adic expansion of positive integers, the $p$-adic expansion of a given partition $\lambda$ exists and is unique.

\begin{theorem}\label{th:Yvertex}
    Let $\lambda$ be a partition of $n$ and $K$ be a field of prime characteristic $p$. The $K\sym{n}$-module $Y^{\lambda}$ has a vertex given by a Sylow $p$-subgroup of $\sym{\rho}$, where $\rho$ is a partition with $|\lambda(j)|$ parts of size $p^j$ for any $j\geq 0$ (and no other parts).
\end{theorem}

We will need \Cref{th:Yvertex} in particular cases.

\begin{lemma}\label{le:p-adic}
    Suppose that $p$ is a prime and $n$ is a positive integer divisible by $p$. If $\mu$ is a partition of $n$ with $n-p <\mu_1 < n$, then the Young module $Y^{\mu}$ has a vertex given by a Sylow $p$-subgroup of $\sym{n-p}$ or $\sym{n-2p}$.
\end{lemma}

\begin{proof}
    Suppose firstly that $\mu_2+n-p \leq \mu_1$. Let $n-p = \sum_{j\geq 0} r_j p^j$ be the usual $p$-adic expansion of $n-p$ with each $r_j < p$. As $p\mid n$, the coefficient $r_0$ is $0$ and the $p$-adic expansion of $\mu$ is $(\mu(0), (r_1), (r_2), \dots)$, where $\mu(0) = (\mu_1-n+p, \mu_2, \mu_3, \dots)$ --- this defines a partition of $p$ distinct from $(p)$ since $\mu_2+n-p \leq \mu_1 < n$; thus is is a $p$-restricted partition. By \Cref{th:Yvertex}, the Young module $Y^{\mu}$ has a vertex given by a Sylow $p$-subgroup of $\sym{\rho}$, where $\rho$ has $r_j$ parts of size $p^j$ for all $j\geq 1$ and no other parts (the parts of size $1$ can be discarded). Since $n-p = \sum_{j\geq 0} r_j p^j$ is the $p$-adic expansion of $n-p$, we conclude that $\sym{\rho}\leq \sym{n-p}$ and $p$ does not divide the index of $\sym{\rho}$ in $\sym{n-p}$; hence, any Sylow $p$-subgroup of $\sym{\rho}$ is a Sylow $p$-subgroup of $\sym{n-p}$.

    We now assume that $\mu_2+n-p > \mu_1$ and let $n-2p = \sum_{j\geq 0} r_j p^j$ be the usual $p$-adic expansion of $n-2p$ with each $r_j < p$. As before, $r_0=0$, however, this time the $p$-adic expansion of $\mu$ is $(\mu(0), (r_1), (r_2), \dots)$, where $\mu(0) = (\mu_1-n+2p, \mu_2, \mu_3, \dots)$ --- this is a partition of $2p$ since $\mu_2\leq n - \mu_1<p<\mu_1-n+2p$ and it is $p$-restricted as all parts apart from the first one are less than $p$ and $\mu_1-n+2p-\mu_2 < p$ by the earlier assumption. Analogously to the first paragraph, we deduce that $Y^{\mu}$ has a vertex given by a Sylow $p$-subgroup of $\sym{n-2p}$.
\end{proof}

 Recall that $M^{\lambda}$ is a direct sum of Young modules $Y^{\mu}$ with $\mu\unrhd\lambda$. Since $Y^{(n)} = M^{(n)}$ is the trivial $K\sym{n}$-module, it is not an indecomposable summand of any Young permutation module of dimension divisible by the characteristic of $K$. We combine this with \eqref{eq:Sformula}, \eqref{eq:Dformula}, \Cref{le:coefficients} and \Cref{le:p-adic}, to prove the final result from \Cref{se:intro}, which we recall here for the convenience of the reader.

\setcounter{section}{1}
\setcounter{theorem}{2}
\begin{theorem}
    Suppose that $K$ is a field of characteristic $p$ and $n\geq 3$ is an integer divisible by $p$. The indecomposable summands of the symmetric powers $\Sym^r S^{(n-1,1)}$ with $2\leq r\leq p-1$ and of the symmetric powers $\Sym^r D^{(n-1,1)}$ with $3\leq r\leq p-1$ have a vertex given by a Sylow $p$-subgroup of $\sym{n-p}$ or $\sym{n-2p}$. 
\end{theorem}

\begin{proof}
    We focus on the symmetric powers of $D^{(n-1,1)}$; the case of the symmetric powers of $S^{(n-1,1)}$ follows along the same lines. By \eqref{eq:Dformula}, the symmetric powers of $D^{(n-1,1)}$ in question are direct sums of Young modules. Let $Y^{\mu}$ be one of them. Then, by \eqref{eq:Dformula}, there is $\lambda\unlhd\mu$ such that $y_r^{\lambda} + y_{r-2}^{\lambda} - 2y_{r-1}^{\lambda}$ is positive and $Y^{\mu}$ is a direct summand of $M^{\lambda}$.
    
    By \Cref{le:coefficients}(iii), we have $n-p < n-r\leq \lambda_1 < n$. Thus also $n-p < \mu_1$. The dimension of $M^{\lambda}$ is the index of $\sym{\lambda}$ in $\sym{n}$, which equals $n!/\lambda!$. Since all parts of $\lambda$ apart from $\lambda_1$ are less than $p$ and $p$ divides $n!/\lambda_1!$, it also divides the dimension of $M^{\lambda}$. By the above discussion, $\mu\neq (n)$, that is $\mu_1 < n$. The result follows from \Cref{le:p-adic}.
\end{proof}
\setcounter{section}{3}
\setcounter{theorem}{14}

\begin{remark}\label{re:Smallvertex}
    The statement of \Cref{th:SDvertex} fails to hold for smaller non-negative integers $r$. This is easy to see using the fact that all indecomposable modules with a vertex given by a Sylow $p$-subgroup of $\sym{n-p}$ have their dimensions divisible by $p$. However, the dimensions of the $r$th symmetric powers of $S^{(n-1,1)}$ with $r=0,1$ and $D^{(n-1,1)}$ with $r=0,1,2$ are not divisible by $p$. 
\end{remark}

\section{Application to calculations of $p$-Kostka numbers}\label{S:kostka}
For a partition $\lambda$ of $n$, we have a decomposition of the Young permutation module $M^\lambda$ into Young modules as follows,
\[
M^\lambda\cong Y^\lambda \oplus \bigoplus_{\mu\rhd\lambda}[M^\lambda:Y^\mu]Y^\mu.
\] The multiplicities $ [M^\lambda: Y^\mu]$ are called \textit{$p$-Kostka numbers}.

Fix a positive integer $n$ divisible by $p$. In this section, we show positivity of some $p$-Kostka numbers using the formulae we obtained in \Cref{ex:SandD}. In particular, we use the equation
\[
[\Sym^4 D^{(n-1,1)}] = 2[M^{(n-2,2)}] + [M^{(n-3,2,1)}] + [M^{(n-4,4)}] - [M^{(n-2,1^2)}] - 2[M^{(n-3,3)}].
\]
 Let $m$ and $n$ be positive integers with p-adic expansions $m=\sum_{i\geq 0}m_ip^i$ and $n=\sum_{i\geq 0}n_ip^i$. We say that \textit{$m$ is $p$-contained in $n$} if $m_i \leq n_i$ for all indices $i \geq 0$.
 The $p$-Kostka numbers $[M^{\lambda} : Y^{\mu}]$ for two-part partitions $\lambda$ and $\mu$ are completely determined in \cite{HENKE}. 
\begin{theorem}\label{pkostka1}
    Suppose $0\leq 2r\leq 2s\leq n$. Then $[M^{(n-s,s)}:Y^{(n-r,r)}]$ is $1$ if $s-r$ is $p$-contained in $n-2r$ and $0$ otherwise. 
\end{theorem}
It follows that if  $p>3$ and $p\mid n$, we have
\begin{align*}
    M^{(n-2,2)}&\cong Y^{(n-2,2)}\oplus Y^{(n-1,1)},\\
    M^{(n-3,3)}&\cong 
          Y^{(n-3,3)}\oplus Y^{(n-2,2)}\oplus Y^{(n-1,1)}
    ,\\
    M^{(n-4,4)}&\cong \left\{ \begin{array}{ll}
          Y^{(n-4,4)}\oplus Y^{(n-3,3)}\oplus Y^{(n-1,1)}& \text{ if $p=5$,}\\
          Y^{(n-4,4)}\oplus Y^{(n-3,3)}\oplus Y^{(n-2,2)}\oplus Y^{(n-1,1)}& \text{ if $p>5$.}\\
    \end{array}\right.
\end{align*}
The $p$-Kostka numbers for the first three-part partition $(n-2,1^2)$ (in the lexicographic order) are determined in \cite{LimWang}. In particular, we use the statement in the case when $p\mid n$.
\begin{theorem}\label{pkostka2}
     Suppose $n\geq 3$ and $p\mid n$ .We have \[M^{(n-2,1^2)}\cong Y^{(n-1,1)}\oplus Y^{(n-2,2)}\oplus Y^{(n-2,1^2)}.\]
\end{theorem}

If $p>5$, we have
\[ [\Sym^4 D^{(n-1,1)}] =
        [M^{(n-3,2,1)}]+[Y^{(n-4,4)}]-[Y^{(n-3,3)}]-[Y^{(n-2,1^2)}].
\]
If $p=5$, we have
\[
[\Sym^4 D^{(n-1,1)}] =
        [M^{(n-3,2,1)}]+[Y^{(n-4,4)}]-[Y^{(n-3,3)}]-[Y^{(n-2,2)}]-[Y^{(n-2,1^2)}].
\]
Since the multiplicity of any Young module on the right side of the equation is nonnegative, we have the following result.
\begin{corollary}\label{co:Kostka}
    Suppose $p>3$ and $p\mid n$. The $p$-Kostka numbers $[M^{(n-3,2,1)}:Y^{(n-3,3)}]$ and $[M^{(n-3,2,1)}:Y^{(n-2,1^2)}]$ are positive. If $p=5$, then $[M^{(n-3,2,1)}:Y^{(n-2,2)}]$ is also positive.
\end{corollary}

\begin{remark}
The $p$-Kostka numbers $[M^{(n-3,2,1)}:Y^{(n-3,3)}]$ and $[M^{(n-3,2,1)}:Y^{(n-2,1^2)}]$ can also be computed using a reduction formula. In \cite{FHKD}, it was proved that the $p$-Kostka numbers are preserved by the first row (respectively, first column) cut. Then in \cite{BOWMAN_GIANNELLI}, it was proved that there is a generalised row (respectively, column) cut formula.
Applying the first row cut, we have $$[M^{(n-3,2,1)}:Y^{(n-3,3)}]=[M^{(2,1)}:Y^{(3)}]=1.
$$
Applying the first column cut, we have $$[M^{(n-3,2,1)}:Y^{(n-2,1^2)}]=[M^{(n-4,1)}:Y^{(n-3)}]=1.
$$
\end{remark}

\subsection*{Acknowledgements} The authors thank Mark Wildon for helpful comments on an earlier version of the manuscript. The first author was supported by an LMS Early Career Fellowship at the University of Birmingham. The second author acknowledges support from EPSRC grant EP/X035328/1.

\appendix

\section{Proof of \Cref{le:Smap}}\label{se:proof}
\begin{proof}[Proof of \Cref{le:Smap}]
We first recall the following notations. Let $s_m$ denote the transposition $(m,m+1)\in \sym{n}$, $e_i=x_i-x_n$, where $x_1, x_2, \dots, x_n$ is the natural permutation basis of $M^{(n-1,1)}$, so $e_1, e_2, \ldots,e_{n-1}$ form a basis for $S^{(n-1,1)}$, and the map $\gamma:\Sym^3 S \to \Sym^2 S$ is given by 
\begin{align*}
    e_i^3&\mapsto -\frac{1}{2}e_i \sum_{l\leq n-1} e_l, \\
    e_i^2e_j&\mapsto \frac{1}{2}(e_ie_j-e_i^2-e_j^2) -\frac{1}{4} \sum_{l\leq n-1} e_l^2, \\
    e_ie_je_k&\mapsto -\frac{1}{4} (e_i^2 +e_j^2 +e_k^2 +\sum_{l\leq n-1} e_l^2)
\end{align*} 
    To show that $\gamma$ is a $K\sym n$-homomorphism, it suffices to show that for $s_m=(m, m+1)$ with $m=1,2,\ldots, n-1$ (which generate $\sym{n}$) and for pairwise distinct $1\leq i,j,k\leq n-1$
    \begin{align*}
        \gamma(s_m\cdot(e_i^3))&=s_m\cdot\gamma(e_i^3)\\
        \gamma(s_m\cdot(e_i^2e_j))&=s_m\cdot\gamma(e_i^2e_j)\\
        \gamma(s_m\cdot(e_ie_je_k))&=s_m\cdot\gamma(e_ie_je_k).
    \end{align*}
    When $1\leq m< n-1$, the identities obviously hold. We are left to check the identities for $m=n-1$. Let $\sigma=s_{n-1}$, so we can write $\sigma\cdot e_i = e_i - e_{n-1}$ for $1\leq i< n-1$ and $\sigma\cdot e_{n-1} =- e_{n-1}$. Note that for $1\leq i,j\leq n-1$, we have $\gamma(e_i^2e_j) = \gamma(e_ie_j^2).$ For $1\leq i <n-1$, we have
    \begin{align*}
        &\gamma(\sigma\cdot e_i^3)\\
        =\;&\gamma((\sigma \cdot e_i)^3)\\
        =\;&\gamma((e_i-e_{n-1})^3)\\
        =\;&\gamma(e_i^3-3e_i^2e_{n-1}+3e_{n-1}^2e_i-e_{n-1}^3)\\
        =\;&-\frac{1}{2}e_i\left(\sum_{l\leq n-1}e_l\right)+\frac{1}{2}e_{n-1}\left(\sum_{l\leq n-1}e_l\right).
    \end{align*}
    On the other hand, notice that $\sum_{l\leq n-1}e_l$ is invariant under the action of $\sym n$ and we have
    \begin{align*}
        &\sigma\cdot\gamma(e_i^3)\\
        =\;&\sigma \cdot \left(-\frac{1}{2}e_i\sum_{l\leq n-1}e_l\right)\\
        =\;&-\frac{1}{2}(e_i-e_{n-1})\left(\sum_{l\leq n-1}e_l\right)\\
        =\;&\gamma(\sigma\cdot(e_i^3)).
    \end{align*}
    In the case where $i=n-1$, we have
    \begin{align*}
        &\gamma(\sigma\cdot e_{n-1}^3)\\
        =\;&\gamma(-e_{n-1}^3)\\
        =\;&\frac{1}{2}e_{n-1}\sum_{l\leq n-1}e_l\\
        =\;&\sigma\cdot\gamma(e_{n-1}^3),
    \end{align*}
    establishing the first identity.
    For $1\leq i\neq j<n-1$, we have
    \begin{align*}
        &\gamma(\sigma\cdot e_i^2e_j)\\
        =\;&
        \gamma((e_i-e_{n-1})^2(e_j-e_{n-1}))\\
        =\;&\gamma(e_i^2e_j-2e_ie_je_{n-1}+e_{n-1}^2e_j-e_i^2e_{n-1}+2e_{n-1}^2e_i-e_{n-1}^3)\\
        =\;&\frac{1}{2}(e_ie_j+e_ie_{n-1}+e_je_{n-1}-2e_i^2-2e_j^2-2e_{n-1}^2)-\frac{3}{4}\sum_{l\leq n-1}e_l^2\\
        &+\frac{1}{2}e_{n-1}\sum_{l\leq n-1}e_l+\frac{1}{2}\left(e_i^2+e_j^2+e_{n-1}^2+\sum_{l\leq n-1}e_l^2\right)\\
        =\;&-\frac{1}{4}\sum_{l\leq n-1}e_l^2+\frac{1}{2}e_{n-1}\sum_{l\leq n-1}e_l-\frac{1}{2}(e_i^2+e_j^2+e_{n-1}^2)\\
        &+\frac{1}{2}(e_ie_j+e_ie_{n-1}+e_je_{n-1}),
    \end{align*}
    and, recalling that $p\mid n$, we also have
    \begin{align*}
        &\sigma\cdot\gamma(e_i^2e_j)\\
        =\;&\sigma\cdot \left(\frac{1}{2} (e_ie_j-e_i^2-e_j^2)-\frac{1}{4}\sum_{l\leq n-1}e_l^2\right)\\
        =\;&\frac{1}{2} ((e_i-e_{n-1})(e_j-e_{n-1})-(e_i-e_{n-1})^2-(e_j-e_{n-1})^2)-\frac{1}{4}\sum_{l< n-1}(e_l-e_{n-1})^2-\frac{1}{4}e_{n-1}^2\\
        =\;&-\frac{1}{4}\sum_{l\leq n-1}e_l^2+\frac{1}{2}e_{n-1}\sum_{l\leq n-1}e_l-\frac{1}{2}(e_i^2+e_j^2+e_{n-1}^2)\\
        &+\frac{1}{2}(e_ie_j+e_ie_{n-1}+e_je_{n-1})\\
        =\;&\gamma(\sigma\cdot e_i^2e_j).
    \end{align*}
If we have $1\leq j< i=n-1$, then
\begin{align*}
        &\gamma(\sigma\cdot e_{n-1}^2e_j)\\
        =\;&\gamma(e_{n-1}^2(e_j-e_{n-1}))\\
        =\;&\frac{1}{2}(e_je_{n-1}-e_j^2-e_{n-1}^2)-\frac{1}{4}\sum_{l\leq n-1}e_l^2+\frac{1}{2}e_{n-1}\sum_{l\leq n-1}e_l,
\end{align*}
and 
\begin{align*}
        &\sigma\cdot \gamma(e_{n-1}^2e_j)\\
        =\;&\sigma\cdot\left(\frac{1}{2}(e_je_{n-1}-e_j^2-e_{n-1}^2)-\frac{1}{4}\sum_{l\leq n-1}e_l^2\right)\\
        =\;&\frac{1}{2}(-(e_j-e_{n-1})e_{n-1}-(e_j-e_{n-1})^2-e_{n-1}^2)-\frac{1}{4}\sum_{l< n-1}(e_l-e_{n-1})^2-\frac{1}{4}e_{n-1}^2\\
        =\;&\frac{1}{2}(e_je_{n-1}-e_j^2-e_{n-1}^2)-\frac{1}{4}\sum_{l\leq n-1}e_l^2+\frac{1}{2}e_{n-1}\sum_{l\leq n-1}e_l\\
        =\;&\gamma(\sigma\cdot e_{n-1}^2e_j).
\end{align*}
If we have $1\leq i<j=n-1$, then
\begin{align*}
        &\gamma(\sigma\cdot e_i^2e_{n-1})\\
        =\;&\gamma(-(e_i-e_{n-1})^2e_{n-1})\\
        =\;&\gamma(-e_i^2e_{n-1}+2e_{n-1}^2e_i-e_{n-1}^3)\\
        =\;&\frac{1}{2}(e_ie_{n-1}-e_i^2-e_{n-1}^2)-\frac{1}{4}\sum_{l\leq n-1}e_l^2+\frac{1}{2}e_{n-1}\sum_{l\leq n-1}e_l\\
        =\;&\sigma\cdot \gamma(e_{n-1}^2e_i)\\
        =\;&\sigma\cdot \gamma(e_i^2e_{n-1}).
\end{align*}

We are left to show the last identity.
Suppose that $1\leq i <j<k<n-1$. We have
\begin{align*}
        &\gamma(\sigma\cdot e_ie_je_k)\\
        =\;&\gamma((e_i-e_{n-1})(e_j-e_{n-1})(e_k-e_{n-1}))\\
        =\;&\gamma(e_ie_je_k-e_ie_je_{n-1}-e_ie_ke_{n-1}-e_je_ke_{n-1}+(e_i+e_j+e_k)e_{n-1}^2-e_{n-1}^3)\\
        =\;&-\frac{1}{4}\left(e_i^2+e_j^2+e_k^2+\sum_{l\leq n-1}e_l^2\right)+\frac{1}{2}((e_i+e_j+e_k)e_{n-1}-e_i^3-e_j^3-e_k^3-3e_{n-1}^2)\\
        &-\frac{3}{4}\sum_{l\leq n-1}e_l^2+\frac{1}{4}\left(2e_i^2+2e_j^2+2e_k^2+3e_{n-1}^2+3\sum_{l\leq n-1}e_l^2\right)+\frac{1}{2}e_{n-1}\sum_{l\leq n-1}e_l\\
        =\;&-\frac{1}{4}\sum_{l\leq n-1}e_l^2+\frac{1}{2}e_{n-1}\sum_{l\leq n-1}e_l-\frac{1}{4}(e_i^2+e_j^2+e_k^2)-\frac{3}{4}e_{n-1}^2+\frac{1}{2}(e_i+e_j+e_k)e_{n-1},
\end{align*}
and we also have
\begin{align*}
        &\sigma\cdot\gamma(e_ie_je_k)\\
        =\;&\sigma\cdot\left(-\frac{1}{4}\left(e_i^2+e_j^2+e_k^2+\sum_{l\leq n-1}e_l^2\right)\right)\\
        =\;&-\frac{1}{4}\left((e_i-e_{n-1})^2+(e_j-e_{n-1})^2+(e_k-e_{n-1})^2+\sum_{l< n-1}(e_l-e_{n-1})^2+e_{n-1}^2\right)\\
        =\;&-\frac{1}{4}\sum_{l\leq n-1}e_l^2+\frac{1}{2}e_{n-1}\sum_{l\leq n-1}e_l-\frac{1}{4}(e_i^2+e_j^2+e_k^2)-\frac{3}{4}e_{n-1}^2+\frac{1}{2}(e_i+e_j+e_k)e_{n-1}\\
        =\;&\gamma(\sigma\cdot e_ie_je_k).
\end{align*}
Finally, suppose that $1\leq i<j<k=n-1$. We have
\begin{align*}
        &\gamma(\sigma\cdot e_ie_je_{n-1})\\
        =\;&\gamma(-(e_i-e_{n-1})(e_j-e_{n-1})e_{n-1})\\
        =\;&\gamma(-e_ie_je_{n-1}+(e_i+e_j)e_{n-1}^2-e_{n-1}^3)\\
        =\;&\frac{1}{4}\left(e_i^2+e_j^2+e_{n-1}^2+\sum_{l\leq n-1}e_l^2\right)+\frac{1}{2}((e_i+e_j)e_{n-1}-e_i^2-e_j^2-2e_{n-1}^2)\\
        &-\frac{1}{2}\sum_{l\leq n-1}e_l^2+\frac{1}{2}e_{n-1}\sum_{l\leq n-1}e_l\\
        =\;&-\frac{1}{4}\sum_{l\leq n-1}e_l^2+\frac{1}{2}e_{n-1}\sum_{l\leq n-1}e_l-\frac{1}{4}(e_i^2+e_j^2)-\frac{3}{4}e_{n-1}^2+\frac{1}{2}(e_i + e_j)e_{n-1},
\end{align*}
and we also have
\begin{align*}
        &\sigma\cdot\gamma(e_ie_je_{n-1})\\
        =\;&\sigma\cdot\left(-\frac{1}{4}\left(e_i^2+e_j^2+e_{n-1}^2+\sum_{l\leq n-1}e_l^2\right)\right)\\
        =\;&-\frac{1}{4}\left((e_i-e_{n-1})^2+(e_j-e_{n-1})^2+e_{n-1}^2+\sum_{l< n-1}(e_l-e_{n-1})^2+e_{n-1}^2\right)\\
        =\;&-\frac{1}{4}\sum_{l\leq n-1}e_l^2+\frac{1}{2}e_{n-1}\sum_{l\leq n-1}e_l-\frac{1}{4}(e_i^2+e_j^2)-\frac{3}{4}e_{n-1}^2+\frac{1}{2}(e_i + e_j)e_{n-1}\\
        =\;& \gamma(\sigma\cdot e_ie_je_{n-1}).
\end{align*}
Next, we prove that $\gamma\circ \mul_2=\id$ which will complete the proof.
We have for $1\leq i\leq n-1$,
\begin{align*}
    &\gamma\circ \mul_2(e_i^2)\\
    =\;&\gamma\left(e_i^2\sum_{l\leq n-1}e_l\right)\\
    =\;&\gamma\left(e_i^3+\sum_{l\neq i}e_i^2e_l\right)\\
    =\;&-\frac{1}{2}e_i\sum_{l\leq n-1}e_l+\sum_{l\neq i}\left(\frac{1}{2}(e_ie_l-e_i^2-e_l^2)-\frac{1}{4}\sum_{t\leq n-1}e_t^2\right)\\
    =\;&-\frac{1}{2}e_i^2 - \frac{n-2}{2}e_i^2 -\frac{1}{2}\sum_{l\leq n-1}e_l^2 + \frac{1}{2}e_i^2 -\frac{n-2}{4}\sum_{l\leq n-1}e_l^2 \\
    =\;&e_i^2.
\end{align*}
For $1\leq i<j\leq n-1$, we have
\begin{align*}
    &\gamma\circ \mul_2(e_ie_j)\\
    =\;&\gamma\left(e_i^2e_j+e_j^2e_i+\sum_{l\neq i,j}e_ie_je_l\right)\\
    =\;&(e_ie_j-e_i^2-e_j^2)-\frac{1}{2}\sum_{l\leq n-1}e_l^2-\frac{1}{4}\sum_{l\neq i,j}\left(e_i^2+e_j^2+e_l^2+\sum_{t\leq n-1}e_t^2\right)\\
    =\;&e_ie_j - (e_i^2 + e_j^2) -\frac{1}{2}\sum_{l\leq n-1}e_l^2 -\frac{n-3}{4}(e_i^2 + e_j^2)\\
    &-\frac{1}{4}\sum_{l\leq n-1}e_l^2 + \frac{1}{4}(e_i^2 + e_j^2) - \frac{n-3}{4}\sum_{l\leq n-1}e_l^2  \\
    =\;&e_ie_j. \qedhere
\end{align*}
\end{proof}

\bibliographystyle{alpha}
\bibliography{bibliography}
\end{document}